\documentclass{article}

\usepackage{arxiv}

\usepackage[utf8]{inputenc} 
\usepackage[T1]{fontenc}    
\usepackage{hyperref}       
\usepackage{url}            
\usepackage{booktabs}       
\usepackage{amsfonts}       
\usepackage{nicefrac}       
\usepackage{microtype}      
\usepackage{lipsum}
\usepackage{amsmath}

\usepackage{amsfonts,amssymb}
\usepackage{graphicx}
\usepackage{amsthm}
\usepackage{xcolor}

\newcommand{\R}{\mathbb{R}}

\newtheorem{theorem}{Theorem}[section]
\newtheorem{proposition}{Proposition}[section]
\newtheorem{lemma}{Lemma}[section]
\newtheorem{corollary}{Corollary}[section]

\newtheorem{definition}{Definition}[section]

\usepackage{cite}

\newcommand{\p}{\partial}
\newcommand{\bb}{\begin{equation}}
\newcommand{\ee}{\end{equation}}
\newcommand{\ba}{\begin{array}}
\newcommand{\ea}{\end{array}}
\newcommand{\f}{\frac}
\usepackage[all]{xy}
\newcommand{\ds}{\displaystyle}

\newcommand{\al}{\alpha}
\newcommand{\be}{\beta}
\newcommand{\de}{\delta}

\newcommand{\sign}{\text{sign}\,}
\numberwithin{equation}{section}

\title{Integrability, existence of global solutions and wave breaking criteria for a generalisation of the Camassa-Holm equation}

\author{
 Priscila~Leal~da Silva\thanks{Corresponding author.} \\
  Centro de Matem\'atica, Computa\c{c}\~ao e Cogni\c{c}\~ao\\
 Universidade Federal do ABC\\
 Santo Andr\'e, Brazil \\
  \texttt{priscila.silva@ufabc.edu.br} \\
  \texttt{pri.leal.silva@gmail.com} \\
   \And
 Igor~Leite~Freire \\
 Mathematical Institute,\\ Silesian University in Opava,\\
Na Rybn\'\i{}\v{c}ku, 1, 74601, Opava, Czech Republic\\
Centro de Matem\'atica, Computa\c{c}\~ao e Cogni\c c\~ao,\\ Universidade Federal do ABC,
\\Santo Andr\'e, SP - Brazil \\
  \texttt{igor.freire@ufabc.edu.br} \\
  \texttt{igor.leite.freire@gmail.com} \\
}

\begin{document}
\maketitle

\begin{abstract}
Recent generalisations of the Camassa-Holm equation are studied from the point of view of existence of global solutions, criteria for wave breaking phenomena and integrability. We provide conditions, based on lower bounds for the first spatial derivative of local solutions, for global well-posedness for the family under consideration in Sobolev spaces. Moreover, we prove that wave breaking phenomena occurs under certain mild hypothesis. Regarding integrability, we apply the machinery developed by Dubrovin [Commun. Math. Phys. 267, 117--139 (2006)] to prove that there exists a unique bi--hamiltonian structure for the equation only when it is reduced to the Dullin--Gotwald--Holm equation. Our results suggest that a recent shallow water model incorporating Coriollis efects is integrable only in specific situations. Finally, to finish the scheme of geometric integrability of the family of equations initiated in a previous work, we prove that the Dullin--Gotwald--Holm equation describes pseudo-spherical surfaces.
\end{abstract}


\section{Introduction}

In the seminal work \cite{CH}, Camassa and Holm considered the equation
\begin{align}\label{CH}
    \begin{aligned}
    m_t &+ um_x + 2u_xm = \alpha u_x, \quad u=u(t,x),
    \end{aligned}
\end{align}
where $(t,x)\in [0,\infty)\times \R$, $m=u-u_{xx}$ and $\alpha\in\R$, and deduced several important properties for it. Equation \eqref{CH}, today known as Camassa-Holm (CH) equation, was firstly obtained in \cite{fokas} and later rediscovered by Camassa and Holm in \cite{CH} as a model for shallow water waves and was shown to be an integrable equation admitting a bi-Hamiltonian formulation and a Lax pair \cite{CHH}. Moreover, Camassa and Holm explicitly presented continuous peaked wave solutions for \eqref{CH} that collide without changing their shapes and speeds. In many of the works that followed \cite{CH}, researchers were interested in understanding the mathematical and physical properties of solutions of \eqref{CH} and were also seeking further generalisations or correlated equations that would share some or even the same properties. Among the generalisations, we mention the Dullin-Gottwald-Holm (DGH) equation
\begin{align}\label{DGH}
    \begin{aligned}
    m_t &+ um_x + 2u_xm = \alpha u_x +\Gamma u_{xxx},\quad \alpha,\Gamma\in\R,
    \end{aligned}
\end{align}
deduced by Dullin, Gottwald and Holm in \cite{DGH1} as a shallow water model in a similar way as \eqref{CH} was derived and proven to be integrable in the same paper.

In the particular case of solutions, well-posedness of \eqref{CH} in Sobolev and similar spaces has been a trending topic and has been widely studied by several authors. Danchin \cite{Dan}, Himonas and Misiolek \cite{HM}, Li and Olver \cite{LO} and Rodrigues-Blanco \cite{RB} have determined local well-posedness for initial data in $H^s(\R)$ with $s>3/2$, while a powerful criteria for global extension of the local solutions was proven by Constantin and Escher in \cite{CE}. In \cite{byers}, Beyers showed that the Cauchy problem to the CH equation is not locally well-posed for any $s<3/2$. Moreover, in the work \cite{HM}, the authors also proved that uniform continuity of solutions for initial data in $H^s(\R)$ with $s<3/2$ does not hold, which makes the value $s=3/2$ critical since there is no continuous embedding between $H^{3/2}(\R)$ and the space of bounded Lipschitz functions and, therefore, \eqref{CH} may not be solved by a standard iterative scheme, see \cite{Dan1}. For this reason, Danchin \cite{Dan1} considered the closely related Besov spaces $B^{3/2}_{2,r}(\R)$ and determined that \eqref{CH} is locally well-posed in $B^{3/2}_{2,1}(\R)$ and ill-posed in $B^{3/2}_{2,\infty}(\R)$. Very recently, Guo, Liu, Molinet and Yiu \cite{GLMY} completed the scheme by determining ill-posedness in $B^{3/2}_{2,2}(\R) = H^{3/2}(\R)$.

Regarding other polynomial Camassa-Holm type equations discovered since \cite{CH}, the Degasperis-Procesi \cite{DP}, DGH \cite{DGH1}, Novikov \cite{HW,Nov} and modified Camassa-Holm \cite{Fok,OR,Qiao} equations arose as integrable, but many other (non-integrable or whose integrability has not been determined) have been discovered, e.g. \cite{anco,DHH,aims,cnmac,HM1}. Our main interest in this paper is to go further with the investigation initiated in \cite{priigorjde} with respect to the equation
\begin{align}\label{cauchy}
    m_t + um_x + 2u_xm = \alpha u_x + \beta u^2u_x+\gamma u^3u_x + \Gamma u_{xxx},
\end{align}
where $m=u-u_{xx}$, $\alpha,\beta,\gamma$ and $\Gamma$ are real arbitrary constants. Equation \eqref{cauchy} clearly includes the CH and DGH equations as particular cases, and a more recent and also physically relevant member of \eqref{cauchy} is
\bb\label{chines-fluid}
\ba{l}
\ds{m_t+um_x+2u_xm+cu_x-\f{\beta_0}{\beta}u_{xxx}+\f{\omega_1}{\al^2}u^2u_x+\f{\omega_2}{\alpha^3}u^3u_x=0},
\ea
\ee
with
\begin{align*}
c&=\ds{\sqrt{1+\Omega^2}-\Omega,\quad\al=\f{c^2}{1+c^2},\quad\beta_0=\f{c(c^4+6c^2-1)}{6(c^2+1)^2},\quad \beta=\f{3c^4+8c^2-1}{6(c^2+1)^2}},\\
\omega_1&=\ds{-\f{3c(c^2-1)(c^2-2)}{2(1+c^2)^3},\quad\omega_2=\f{(c^2-1)^2(c^2-2)(8c^2-1)}{2(1+c^2)^5}},
\end{align*}

where $c$ relates to the constant rotational frequency due to Coriolis effect. The latter equation has been recently proposed and considered in \cite{chines-adv,gui-jnl,GLS,chines-jde}, see also \cite{priigorjde,igor}.

Regarding equation \eqref{cauchy}, observe that by taking 
\begin{align}\label{h}
h(u) = (\alpha+\Gamma) u + \frac{\beta}{3}u^3+ \frac{\gamma}{4}u^4
\end{align}
it can be rewritten as
\begin{align}\label{gCH}
    m_t + (u+\Gamma)m_x + 2 u_xm = \partial_x h(u).
\end{align}

Among the results of our previous paper \cite{priigorjde} with respect to \eqref{cauchy}, we provided a proof of local well-posedness in Sobolev spaces. More explicitly, we showed that if the initial data $u_0=u_0(x)$ is in $H^s(\R)$ for $s>3/2$, then there exists a maximal time of existence $T>0$, also called \textit{lifespan}, such that the solution $u\in C([0,T),H^s(\R))\cap C^1([0,T),H^{s-1}(\R))$ exists and is unique. Although this lifespan $T$ guarantees the existence of a local solution, the result nothing says about how $u$ behaves beyond $T$.

In this paper, one of our goals is to investigate how the solution behaves as time reaches the lifespan $T$. It is said that a solution $u$ is global if $T=\infty$, and in the case $T<\infty$ the solution is said to blow-up at a finite time. 

\begin{theorem}\label{teo1.1}
Given $u_0$ in $H^s(\R)$, with $s>3/2$, let $u$ be the unique solution of the Cauchy problem of \eqref{cauchy} and $m=u_0-u''_{0}$. If $m_0\in L^1(\R)\cap H^1(\R)$ does not change its signal and $\sign{(m)}=\sign{(m_0)}$, then $u$ exists globally.
\end{theorem}

One of the problems of Theorem \ref{teo1.1} is the requirement that $m_0$ does not change sign. It is possible, however, to remove this condition and allow it to change sign once and still be able to guarantee existence of global solutions, given by Theorem \ref{teo1.2}.

\begin{theorem}\label{teo1.2}
Let $u_0\in H^s(\R)$ for $s>3/2$, $m_0=u_0-u_0''$ and $u$ be the unique corresponding solution of \eqref{cauchy}. If $\sign{(m)}=\sign{(m_0)}$ and there is $x_0\in\R$ such that $m_0(x)\leq 0$ if $x\leq x_0$ and $m_0(x)\geq 0$ for $x\geq x_0$, then $-\Vert u_0\Vert_{H^1}\leq u_x(t,x)$. In particular, the solution $u$ does not blow-up at a finite time.
\end{theorem}

Given the possibility of extending local to global solutions, it is natural to question under which conditions, if any at all, a solution blows-up. In case of existence of such conditions, we may be tempted to search for the occurrence of wave breaking. In Physics, a wave breaking is a phenomena in which the wave energy becomes infinite and its crest overturns. A mathematical explanation for it is that if we section a wave and observe the behaviour of tangent lines, the wave breaking is characterized by a vertical tangent line when time reaches the lifespan $T$.

Formally speaking, we call a finite time blow-up a wave breaking if
\begin{align*}
    \sup\limits_{(t,x)\in[0,T)\times \R}\vert u(t,x)\vert <\infty, \quad \limsup \{\sup\vert u_{x}(t,x)\vert\} =\infty.
\end{align*}
As we will see in Proposition \ref{prop2.3}, a necessary condition for wave breaking of \eqref{cauchy} is the unboundedness of the first spatial derivative $u_x$ in $[0,T)\times \R$. Moreover, in Section \ref{Sec3} we will introduce a parameter that will reduce our analysis of wave breaking of \eqref{cauchy} to a problem similar to the Camassa-Holm equation \eqref{CH}, see \cite{Escher1}. Regarding wave breaking, our main result is given by the following theorem:

\begin{theorem}\label{teo1.3}
Given $u_0 \in H^{s}(\R)$, with $s>3/2$, let $u\in C^1([0,T),H^{s-1}(\R))$ be the corresponding solution of \eqref{cauchy}. Additionally, suppose there exists $x_0\in\R$ such that
\begin{align*}
    \sqrt{2\sigma} u_0'(x_0)< \min \{-\Vert u_0\Vert_{H^1},-\Vert u_0\Vert_{H^1}^{p/2}\},
\end{align*}
where $\sigma\in\R$ is such that $0<\sigma \leq \displaystyle{\frac{1}{1+36K}}$, $K=4\max\{\vert \alpha\vert, \vert \beta\vert/3, \vert\gamma\vert/4,\vert \Gamma \vert\}$ and $p$ is such that
\begin{align*}
    \Vert u_0\Vert_{H^1}^p=\max\{\Vert u_0\Vert_{H^1},\Vert u_0\Vert_{H^1}^3,\Vert u_0\Vert_{H^1}^4\}.
\end{align*}
Then $u$ breaks at finite time.
\end{theorem}

After the work of Camassa and Holm \cite{CH}, the study of integrable equations was turned upside down. In fact, until the 90's most efforts made in the classification of integrable partial differential equations were dedicated to evolutionary equations such as the Korteweg-de Vries equation or hyperbolic ones as the Sine-Gordon equation. The integrability of the Camassa-Holm equation \eqref{CH} was a breakthrough discovery due to the nonlocal specificity of its evolution form \eqref{CH}. As we shall see, the inverse of the operator $1-\partial_x^2$ plays an important role in the analysis of most properties of \eqref{CH} and can be the source of many theoretical issues.

Integrability is a somewhat rare property of a partial differential equation and several definitions (related but not equivalent) can be given for it, see, for instance, \cite{ablo1,ablo2,sok,olverbook}. In this paper, equation \eqref{cauchy} will be said to be integrable if it admits two compatible Hamiltonian formulations\footnote{The $-$ sign is unnecessary, but convenient in our particular case, see Section \ref{Sec4}.} 
\begin{align}
    m_t = -B_1\f{\delta {\cal H}_1}{\delta u} =- B_2\f{\delta {\cal H}_2}{\delta u},
\end{align}
where $B_1$ and $B_2$ are the Hamiltonian operators satisfying the Jacobi identity, ${\cal H}_1$ and ${\cal H}_2$ are the Hamiltonian functions such that any linear combination of Hamiltonian operators is still a Hamiltonian operator. An important consequence of the existence of such a formulation is that it is sufficient to guarantee the existence of an infinite hierarchy of linearly independent conservation laws.

It is well known that equation \eqref{cauchy} is integrable if $\beta$ and $\gamma$ vanishes and the resulting equation becomes the Dullin-Gotwald-Holm equation, see \cite{DGH1,DGH2}, but regarding arbitrary choices of $\beta$ and $\gamma$, the only result known so far shows that \eqref{cauchy} describes pseudo-spherical surfaces (geometric integrability) if $\beta=\gamma=\Gamma=0$, \textit{suggesting} integrability only for the Camassa-Holm equation, see \cite{priigorjde}. It is important to emphasize that $\Gamma=0$ was not necessarily required\footnote{In \cite{priigorjde} we replaced $u$ by $u-\Gamma$ in \eqref{cauchy}, which is equivalent to take $\Gamma=0$ in \eqref{cauchy}. Therefore, what we have done does not preserve initial data.} in \cite{priigorjde} for integrability and, therefore, geometric integrability should be recovered for $\Gamma\neq 0$. In Section \ref{Sec4} we finish the scheme for geometric integrability of equation \eqref{cauchy} and determine that we can eliminate the restriction $\Gamma=0$ in \eqref{cauchy} for it to describe pseudo-spherical surfaces.

\begin{theorem}\label{teo1.4}
    Equation \eqref{DGH} describes pseudo-spherical surfaces with associated triplet given by
    \begin{align*}
    \theta_1&=\ds{\left(m+b\right)dx-\left[(u+\Gamma)m +(b+1)u+b(\Gamma+1)\mp\eta u_x \right]}dt,\\
    \theta_2&=\ds{\eta\,dx-[\eta(1+u+\Gamma)\mp u_x]dt},\\
    \theta_3&=\ds{\pm\left(m+b+1\right)dx+\left[\eta u_x\mp (u+\Gamma)(m+1)\mp(u+1)(b+1)\mp\Gamma b\right]dt,}
    \end{align*}
where $\eta^2=2+2b+\al+\Gamma$. In particular, the DGH equation is geometrically integrable.
\end{theorem}

Although the DGH equation is very similar and shares several properties with the CH equation, we have not found any result proving that the DGH equation describes pseudo-spherical surfaces, although such result has been known for the CH equation for nearly two decades \cite{reyes2002}.

We also apply a perturbative method due to Dubrovin \cite{dubrovincpam} to investigate other integrability properties of equation \eqref{cauchy}. In fact, Dubrovin's work motivates us to propose the following weak notion of integrability:
\begin{definition}\label{new}
An equation is said to be quasi-integrable if, eventually under a change of variables, it satisfies lemmas \ref{lema4.1} and \ref{lema4.2}.
\end{definition}

Lemmas \ref{lema4.1} and \ref{lema4.2}, presented in Subsection \ref{bi-ha}, essentially show the existence of an infinite hierarchy of {\it approximate symmetries} \cite{ibragimov,dubrovincmp} and uniqueness of a bi-Hamiltonian deformation of hyperbolic equations. A key step to apply Dubrovin's machinery is to make a change of variables so that the higher order terms of \eqref{cauchy} can be seen as perturbation of a hyperbolic equation. In our case, we have 
\bb\label{4.1.5}
u_t-\epsilon^2u_{txx}-\left(\f{3}{2}uu_x+\al u_x+\be u^2u_x+\gamma u^3u_x\right)+\epsilon^2\left(u_xu_{xx}+\f{1}{2}uu_{xxx}-\Gamma u_{xxx}\right)=0.
\ee
Equation \eqref{4.1.5} can be transformed (see Section \ref{bi-ha}) into an evolution equation. We have the following result about integrability.
\begin{theorem}\label{teo1.5}
Equation \eqref{4.1.5} is quasi-integrable if and only if $\be=\gamma=0$. In particular, its perturbed Hamiltonian is
$$
{\cal H}=\int \left(\f{1}{4}v^3+\f{\al}{2}v^2-\epsilon^2\f{v+\al+\Gamma}{3}v_x^2+\epsilon^4\f{v+\al+\Gamma}{2}v_{xx}^4\right)dx,\quad v=(1-\epsilon^2\p_x^2)u.
$$
Moreover, equation \eqref{4.1.5} admits a unique bi-Hamiltonian structure given by the deformations
$$
\{v(x),v(y)\}_1=\delta'(x-y),\quad \{v(x),v(y)\}_2=(v(x)+\al+\Gamma)\delta'(x-y)+\f{1}{2}v_x\delta(x-y).
$$
\end{theorem}

\begin{corollary}
Equation \eqref{4.1.5} is quasi-integrable if and only if it is equivalent to the DGH equation.
\end{corollary}

The paper is organized as follows: in Section \ref{Sec2} we present notation and prove Theorem \ref{teo1.1} with the help of auxiliary results that guarantee the existence of global solutions by providing lower bounds for the first spatial derivative of the local solution. Moreover, we also prove Theorem \ref{teo1.2} in Section \ref{Sec2}. In Section \ref{Sec3} we prove Theorem \ref{teo1.3} by obtaining conditions for lower unboundedness of $u_x$ and also by introducing a real parameter that enables us to reduce the criteria to one similar to the CH equation. In Section \ref{Sec4} we apply the machinery of Dubrovin \cite{dubrovincmp} to prove that equation \eqref{cauchy} is quasi-integrable if (and only if) it is reduced to the DGH equation. Furthermore, in Section \ref{Sec4} we complete the scheme of geometric integrability of \eqref{cauchy} by showing that the DGH equation describes pseudo-spherical surfaces. Finally, in Section \ref{Sec5} we present a discussion of our achievements and how they impact some recent results found in the literature.

\section{Global well-posedness}\label{Sec2}

In this section we will prove Theorem \ref{teo1.1}. More explicitly,  we will go through propositions that use the local solution of \eqref{cauchy} and some auxiliary functions and estimates to develop the conditions to extend the solution to a global one in $H^2(\R)$. Then we will use a common density argument to conclude Theorem \ref{teo1.1}. But before proceeding with the results, we will present some definitions of spaces and technicalities required for the complete understanding of what follows.

Let $L^2(\R)$ denote the space of real square integrable functions $f:\R\rightarrow \R$. It is a Hilbert space when endowed with the inner product
\begin{align*}
\langle f\,,\, g\rangle = \int_{\R} f(x)g(x)dx.
\end{align*}
Moreover, we define the Banach space $L^{\infty}(\R)$ as the set of functions $f$ such that
\begin{align*}
\Vert f\Vert_{L^{\infty}(\R)} :=\text{esssup}\vert f(x)\vert<\infty.
\end{align*}
Let $C^{\infty}_0(\R)$ be the set of smooth functions with compact support and $S(\R)$ be the set of smooth functions decaying to $0$ at infinity, with the same property being held by any of their derivatives. Then $C^{\infty}_0(\R)$ is dense in $S(\R)$. The space $S(\R)$ is called Schwarz space and its elements are called \textit{test functions}.

The elements of the dual space $S'(\R)$ are called \textit{tempered distributions} and, given $\phi\in S'(\R)$, we define its Fourier transform by
\begin{align*}
{\cal F}(\phi)(\xi)=\hat{\phi}(\xi) := \frac{1}{\sqrt{2\pi}}\int_{\R}\phi(x)e^{-ix\xi}dx,
\end{align*}
with its inverse being given by
\begin{align*}
\phi(x) = {\cal F}^{-1}(\hat{\phi})(x)= \frac{1}{\sqrt{2\pi}}\int_{\R}\hat{\phi}(x)e^{ix\xi}d\xi.
\end{align*}
Given $s\in\R$, the set of tempered distributions $f\in S'(\R)$ such that $(1+|\xi|^2)^{s/2}\hat{f}(\xi)\in L^2(\R)$, denoted by $H^s(\R)$ and called Sobolev space, is a Hilbert space when endowed with the inner product of $L^2(\R)$. The following embedding properties of $H^s(\R)$ will be constantly used throughout this paper (see Chapter 4 in \cite{taylor}):
\begin{itemize}
\item If $s\geq t$, then $S(\R)\subset H^s(\R)\subset H^t(\R) \subset S'(\R)$.
\item If $s>1/2$, then $H^s(\R)\subset L^{\infty}(\R)$.
\item For each $s\in \R$, let $\Lambda^su={\cal F}^{-1}((1+|\xi|^2)^{s/2}\hat{u})$. Then for all $s$ and $t$, $\Lambda^s:H^t\rightarrow H^{t-s}$ is an isomorphism and its inverse is denoted by $\Lambda^{-s}$. In particular we have $H^s(\R) = \Lambda^{-s}(L^2(\R))$ and $\Vert u \Vert_{H^s(\R)} = \Vert \Lambda^s u\Vert_{L^2(\R)}.$
\end{itemize}

Regarding the isomorphism $\Lambda^s$, we mention that the most important choice in this paper is $s=2$. In particular, $\Lambda^2$ can be identified with the Helmholtz operator $1-\partial_x^2$, while its inverse is given by $\Lambda^{-2}f = \frac{e^{-\vert x\vert}}{2} \ast f$, where $\ast$ denotes convolution operation.

Given an interval $I=[0,a)$, with $a>0$, let $f$ be a continuous function on $I$ with continuous first derivative in the interior of $I$ and let $k\in\R$ a constant. If $f'(t)\leq kf(t)$, then
\begin{align}\label{2.1}
f(t) \leq f(0)e^{kt},
\end{align}
for all $t\in[0,a)$. The expression \eqref{2.1} is known as Gronwall's inequality and the proof of a more general inequality can be found in \cite{hunter}, page 56.

One important result proven in \cite{priigorjde} is that if any solution of \eqref{cauchy} and its spatial derivatives vanish at infinity, then its $H^1$-norm is conserved and we have
\begin{align}\label{conservation}
    \Vert u \Vert_{H^1} = \Vert u_0\Vert_{H^1}
\end{align}
for $u_0=u_0(x):=u(0,x)$.

The following lemma will be important for the results of this section.

\begin{lemma}\label{lem2.1}
Let $u$ be a solution of \eqref{cauchy} such that $u_{tx}=u_{xt}$ and $u,u_x,u_{xx}$ are integrable and vanishing at infinity. If $m_0:=u_0-u_0''$, then
\begin{align*}
\int_{\R}m_0dx = \int_{\R} mdx = \int_{\R}udx= \int_{\R}u_0dx.
\end{align*}
\end{lemma}
\begin{proof}
See \cite{priigorjde}, Theorem 4.2 and its Corollary.
\end{proof}


In the next result we prove existence and uniqueness of solutions for the auxiliary Cauchy problem \begin{align}\label{aux}
    \begin{aligned}
        \begin{cases}
        \ds{\f{d}{dt}q(t,x) = u(t,q) + \Gamma},\\
        \\
        q(0,x)=x.
        \end{cases}
    \end{aligned}
\end{align}
that depends on the solution of the Cauchy problem of  \eqref{cauchy}. For the propositions that follow, we will consider $s=3$ and use local well-posedness of \eqref{cauchy} in $H^2(\R)$.

\begin{proposition}\label{prop2.1}
Given $u_0\in H^3(\R)$, let $u\in C^1([0,T),H^2(\R))$ be the solution of \eqref{cauchy}. Then the initial value problem \eqref{aux} has a unique solution $q(t,x)$ and $q_x(t,x)>0$ for any $(t,x)\in[0,T)\times \R$. Moreover, for each $t\geq 0$ fixed, $q(t,\cdot)$ is a diffeomorphism on the line.
\end{proposition}
\begin{proof}
Since $u\in C^1([0,T),H^2(\R))$ and $H^2(\R)\subset C^1(\R)$, we have $u\in C^1([0,T)\times \R; \R)$, and by means of traditional existence theorems we guarantee the existence of a unique solution $q(t,x)$.

Differentiation of \eqref{aux} with respect to $x$ yields
\begin{align}\label{aux1}
\begin{aligned}
\begin{cases}
\ds{\frac{d}{d t}q_x(t,x)=u_x(t,q)q_x(t,x)},\\
\\
q_x(0,x)=1,
\end{cases}
\end{aligned}
\end{align}
where $u_x(t,q)$ denotes partial derivative with respect to the second variable. The solution of \eqref{aux1} is then explicitly given by
\begin{align*}
q_x(t,x) = \exp \left(\int_0^t u_x(s,q(s,x))ds\right)>0.
\end{align*}
The proof that $q(t,\cdot)$ is a diffeomorphism can be found in \cite{const2000-1} (Theorem 3.1).
\end{proof}

We shall now consider equation \eqref{cauchy} written in the form \eqref{gCH} with the function $h(u)$ given by \eqref{h}.


\begin{proposition}\label{prop2.2}
Given $u_0\in H^3(\R)$, let $u\in C^1([0,T),H^2(\R))$ be the solution of \eqref{cauchy} and $m_0=u_0-u_0''$. Then
\begin{align*}
mq_x^2 = m_0+\int_0^t q_x^2(s,x)\partial_x h(u(s,q))ds,
\end{align*}
where $q$ is the solution of \eqref{aux} and $h(u)$ is given by \eqref{h}.
\end{proposition}
\begin{proof}
Differentiating the term $mq_x^2$ with respect to $t$, we obtain
\begin{align*}
\frac{d}{dt} (mq_x^2) &= m_tq_x^2 + m_xq_tq_x^2 + 2mq_xq_{tx}\\
&=(\partial_xh(u) - (u+\Gamma)m_x-2u_xm)q_x^2 + 2u_xmq_x^2 + (u+\Gamma)m_xq_x^2 \\
&= (\partial_xh(u))q_x^2.
\end{align*}
Integrating form $0$ to $t$ yields
\begin{align*}
mq_x^2 = m_0 + \int_0^t q_x^2(s,x)\partial_xh(u(s,x))ds
\end{align*}
and the result is proven.
\end{proof}

Before proceeding with the next proposition we need a technical result.

\begin{lemma}\label{lem2.2a}
Let $F$ and $f$ be functions such that $F\in C^\infty(\R)$, $F(0)=0$, and $f\in H^{s}(\R)$, with $s>1/2$. If, for some $t>1/2$, the $H^t(\R)$ norm of $f$ is bounded from above by $R>0$, that is $\|f\|_{H^t(\R)}<R$, then there exists a positive constant $c$ depending only on $R$ such that $F(f)\in H^s(\R)$ and $\|F(f)\|_{H^s(\R)}\leq c\|f\|_{H^s(\R)}$.
\end{lemma}
\begin{proof}
The proof that $F(f)$ belongs to $H^s(\R)$ follows Lemma 3 of \cite{const-mol}. Hence, it remains to prove the existence of $c>0$ such that $\|F(f)\|_{H^s(\R)}\leq c\|f\|_{H^s(\R)}$.

Since $t>1/2$, again by the Sobolev Embedding Theorem we have the inequality $\|f\|_{L^\infty(\R)}\leq c_1\|f\|_{H^t(\R)}\leq c_1\,R$, for some positive constant $c_1$, depending only on $t$, but not uppon $f$. Moreover, we have
$$
F(f(x))=\int_0^1f(x)F'(tf(x))dt\Rightarrow |F(f(x))|\leq |f(x)|\sup_{|y|\leq c_1 R}|F'(y)|=:c|f(x)|.
$$

Due to the fact that $|F(f)|$ is bounded from above by $c|f(x)|$, consequently the same remains valid to their corresponding norms.
\end{proof}

\begin{corollary}\label{cor2.1}
Let $u_0\in H^s(\R)$, $s>3/2$, and $u$ be the corresponding (local) solution of \eqref{cauchy} with initial condition $u(0,x)=u_0(x)$. Then there exists a constant $c=c(\|u_0\|_{H^1(\R)})>0$ such that $\|h(u)\|_{H^s(\R)}\leq c\|u\|_{H^s(\R)}$, where $h$ is the function given by \eqref{h}.
\end{corollary}
\begin{proof}
We firstly observe that both $h$ and $u$ satisfy the conditions required by Lemma \ref{lem2.2a} in view of the results proved in \cite{priigorjde} and the Proposition on page 1065 of \cite{const-mol}. In particular, we have $\|u\|_{H^1(\R)}=\|u_0\|_{H^1(\R)}$. The result then follows from Lemma \ref{lem2.2a}, with $t=1$ and $R=\|u_0\|_{H^1(\R)}$.
\end{proof}

The next proposition gives a condition for existence of global solutions.

\begin{proposition}\label{prop2.3}
Given $u_0\in H^3(\R)$, let $u$ be the solution of \eqref{cauchy}. If $u,u_x,u_{xx},u_{xxx}$ vanish at infinity and there exists $\kappa>0$ such that $u_x>-\kappa$, then there is a differentiable function $\sigma(t)$ such that $\Vert u\Vert_{H^3}\leq \sigma \Vert m_0\Vert_{H^1}$. In particular, $u$ does not blow up at a finite time.
\end{proposition}
\begin{proof}
Let $\langle\cdot\,,\, \cdot\rangle$ denote the usual inner product in $L^2(\mathbb{R})$. Then
$$\frac{d}{dt}\Vert m\Vert^2_{L^2}= 2\langle m_t\,,\, m\rangle.$$
From \eqref{gCH}, we have $m_t=\partial_x h(u) - (u+\Gamma)m_x - 2u_xm$, which yields
$$\langle m_t\,,\, m\rangle = \langle \partial_xh(u)\,,\,m\rangle - \langle um_x\,,\, m\rangle - \langle \Gamma m_x\,,\,m\rangle - 2\langle u_xm\,,\, m\rangle.$$
Observing that $\langle um_x\,,\, m \rangle = -\frac{1}{2} \langle u_x\,,\, m^2 \rangle, \langle u_xm\,,\, m\rangle = \langle u_x\,,\, m^2\rangle,$ the condition
$$\frac{d}{dt}\Vert m\Vert^2_{L^2}= -3\langle u_x\,,\,m^2\rangle - 2\Gamma\langle m_x\,,\, m\rangle +2 \langle \partial_x h(u)\,,\,m\rangle$$
is obtained. For the term $\Vert m_x\Vert^2$, we calculate the derivative of \eqref{gCH} with respect to $x$ to obtain
$$m_{tx} =\partial_x^2 h(u) - (u+\Gamma)m_{xx} - 3u_xm - 2u_{xx}m,$$
which means that
$$\langle m_{tx}\,,\, m_x \rangle = \langle \partial_x^2h(u)\,,\, m_x \rangle- \langle um_{xx}\,,\, m_x\rangle- \langle \Gamma m_{xx}\,,\, m_x\rangle - 3 \langle u_xm_x\,,\, m_x\rangle-2\langle u_{xx}m\,,\, m_x\rangle.$$
Given that $\langle (u+\Gamma)m_{xx}\,,\, m_x\rangle = -\frac{1}{2}\langle u_x\,,\, m_x^2\rangle, \langle u_xm_x\,,\, m_x\rangle = \langle u_x\,,\,m_x^2 \rangle$ and
$$0 = \frac{1}{3}\int_{\R}\partial_x m^3dx = \int_{\R} m^2m_xdx = \langle u\,,\, mm_x\rangle - \langle u_{xx}m\,,\, m_x\rangle,$$
the derivative of the term $\Vert m_x\Vert_{L^2}^2$ is written as
$$\frac{d}{dt} \Vert m_x\Vert^2_{L^2} = 2\langle m_{tx}\,,\, m_x\rangle=-5 \langle u_x\,,\,m_x^2 \rangle - 4 \langle u\,,\, mm_x\rangle + 2\langle F\,,\, m_x\rangle,$$
where $F:=\partial_x^2 h(u) - \Gamma m_{xx}$. Therefore,
\begin{align*}
\frac{d}{dt}\Vert m\Vert^2_{H^1} = &\frac{d}{dt}(\Vert m\Vert^2_{L^2} +\Vert m_x\Vert^2_{L^2})\\
=& -3\langle u_x\,,\, m^2\rangle - 5 \langle u_x\,,\, m_x^2\rangle - 4 \langle u\,,\, mm_x\rangle + 2 \langle (1-\partial_x)\partial_x h(u)\,,\,m \rangle - 2\Gamma \langle m_{xx}-m\,,\,m_x \rangle\\
=& -3\langle u_x\,,\, m^2\rangle - 5 \langle u_x\,,\, m_x^2\rangle - 4 \langle u\,,\, mm_x\rangle + 2 \langle (1-\partial_x)\partial_x h(u)\,,\,m \rangle,
\end{align*}
where in the last equality we used the fact that $u,u_x,u_{xx},u_{xxx}$ vanish at inifinity to obtain $\langle m_{xx}-m\,,\,m_x \rangle=0$. Since $u_x>-\kappa$, we have $-\langle u_x\,,\, m^2\rangle <\kappa \Vert m\Vert_{L^2}^2$ and $-\langle u_x\,,\, m_x^2\rangle < \kappa \Vert m_x\Vert_{L^2}^2$. Moreover, from the Sobolev embedding theorem with $s=1$, Lemma 1 of \cite{mustafa} and \eqref{conservation}, the estimate $\Vert u\Vert_{L^{\infty}}\leq \Vert u\Vert_{H^1}=\Vert u_0\Vert_{H^1}$ holds. Then
\begin{align*}
-4 \langle u\,,\, mm_x\rangle \leq& 4 \langle \vert u\vert\,,\, \vert m\vert \vert m_x \vert\rangle \leq 4\Vert u\Vert_{L^{\infty}} \langle \vert m\vert\,,\, \vert m_x\vert \rangle\\
\leq& \Vert u\Vert_{L^{\infty}}\Vert m\Vert^2_{L^2}\Vert m_x\Vert_{L^2}^2 \leq 2\Vert u\Vert_{L^{\infty}}(\Vert m\Vert^2_{L^2} + \Vert m_x\Vert_{L^2}^2)\\
\leq& 2\Vert u_0\Vert_{H^1} \Vert m\Vert^2_{H^1},
\end{align*}
where in the last inequality we used the algebraic condition $2ab\leq a^2+b^2$. Therefore, there exists a positive constant $C$ depending only on $\Vert u_0\Vert_{H^1}$ such that 
\begin{align*}
\frac{d}{dt} \Vert m\Vert^2_{H^1} \leq& C \Vert m\Vert_{H^1} + 2 \langle (1-\partial_x^2)\partial_xh(u)\,,\, m\rangle.
\end{align*}
From the Cauchy-Schwarz inequality and the same algebraic condition, we have
\begin{align*}
\langle (1-\partial_x^2)\partial_xh(u)\,,\, m\rangle\leq \frac{1}{2}(\Vert (1-\partial_x^2)\partial_x h(u)\Vert^2_{L^2} + \Vert m \Vert_{L^2}^2).
\end{align*}
However, since $h(0)=0$, Lemma \ref{lem2.2a} and Corollary \ref{cor2.1} guarantee the existence of a positive constant $c_0$ such that $\Vert h(u)\Vert_{H^3}\leq c_0\Vert u\Vert_{H^3}$ . Therefore, 
$$\Vert (1-\partial_x^2)\partial_xh(u)\Vert_{L^2} = \Vert \partial_xh(u)\Vert_{H^2} \leq \Vert h(u)\Vert_{H^3} \leq c_0 \Vert u\Vert_{H^3} = c_0 \Vert m\Vert_{H^1}.$$

We can then construct another positive constant $c_1$ such that
\begin{align*}
\frac{d}{dt} \Vert m\Vert_{H^1}^2 \leq c_1 \Vert m\Vert^2_{H^1},
\end{align*}
which after simplification yields
\begin{align*}
\frac{d}{dt} \Vert m\Vert_{H^1} \leq c_1 \Vert m\Vert_{H^1}.
\end{align*}
Finally, Gronwal's inequality tells that
\begin{align*}
\Vert m\Vert_{H^1} \leq \sigma(t) \Vert m_0\Vert_{H^1},
\end{align*}
where $\sigma(t)=e^{c_1t}.$ The conclusion of the theorem follows from
\begin{align*}
\Vert u\Vert_{H^3} = \Vert \Lambda^2u\Vert_{H^1} = \Vert m\Vert_{H^1}\leq \sigma(t)\Vert m_0\Vert_{H^1}.
\end{align*}
\end{proof}

Due to Proposition \ref{prop2.3}, for the existence of global solutions of \eqref{cauchy} it is enough to find a lower bound for the first $x$ derivative. The next lemma is dedicated to this purpose.

\begin{lemma}\label{lem2.2}
Let $u$ be the solution of \eqref{cauchy} associated with an initial data $u_0\in H^3(\R)$ and $m_0=u_0-u_0''$. Assume that $m_0\in L^1(\R)\cap H^1(\R)$, its sign does not change and $\text{sign}(m) = \text{sign}(m_0)$. Then
\begin{enumerate}
\item the signs of $u$ and $u_0$ do not change and $\text{sign}(u) = \text{sign}(u_0)$;
\item for any $(t,x)\in [0,T)\times \R$ we have $u_x\geq -\Vert m_0\Vert_{L^1}.$
\end{enumerate}
\begin{proof}
See \cite{priigorjde}, Corollary 4.2.
\end{proof}
\end{lemma}
\noindent\textbf{Proof of Theorem \ref{teo1.1}:} We are now in conditions to prove Theorem \ref{teo1.1} for $u\in C^1([0,T),H^2(\R))$. For $u_0\in H^3$, let $u$ be the unique local solution of \eqref{cauchy}. Assume that $m_0\in L^1(\R)\cap H^1(\R)$ is such that $\sign{(m)}=\sign{(m_0)}$. From Lemma \ref{lem2.2} and Proposition \ref{prop2.3} we conclude that $u$ does not blow up at finite time and, consequently, it exists globally in $H^2(\R)$.

We shall extend the argument to $s>3/2$ by density. Since the solution exists globally as a function of $x$ in $H^2(\R)$, for $s\geq 2$ we have $H^s(\R)\subset H^2(\R)$ and the property of $u$ being global also holds for $H^s(\R)$. If $s\in(3/2,2]$, then $H^2(\R)\subset H^s(\R)$ and the global solution $u$ belongs to $H^s(\R)$. That being so, Theorem \ref{teo1.1} is proven for $s>3/2$.\hfill$\square$

\textbf{Proof of Theorem \ref{teo1.2}}: If $\sign{(m)}=\sign{(m_0)}$, then  $\text{sign}\,u=\text{sign}\,m$ in view of the relation $u=\f{e^{-|x|}}{2}\ast m$. In fact,
$$
u(t,x)=\f{1}{2}\int_\R e^{-|x-\xi|}m(t,\xi)d\xi=\f{1}{2}e^{-x}\int_{-\infty}^xe^\xi m(t,\xi)d\xi+\f{1}{2}e^{x}\int^{\infty}_xe^{-\xi} m(t,\xi)d\xi.
$$
Calculating $u_x$, we have
$$
u_x(t,x)=-\f{1}{2}e^{-x}\int_{-\infty}^xe^\xi m(t,\xi)d\xi+\f{1}{2}e^{x}\int^{\infty}_xe^{-\xi} m(t,\xi)d\xi.
$$
Since $\text{sign}\,m=\text{sign}\,m_0$, then $m(t,x)\leq 0$ if $ x\leq q(t,x_0)$ and $m(t,x)\geq0$ if $x\geq q(t,x_0)$, where $q$ is the function in \eqref{aux}. Therefore,
$$
\ba{lcl}
u_x(t,x)&=&\ds{-\f{1}{2}e^{-x}\int_{-\infty}^x e^\xi m(t,\xi)d\xi-\f{1}{2}e^{x}\int_x^\infty e^{-\xi}m(t,\xi)d\xi}\\
\\
&&\ds{+\f{1}{2}e^{x}\int_x^\infty e^{-\xi}m(t,\xi)d\xi+\f{1}{2}e^{x}\int^{\infty}_xe^{-\xi} m(t,\xi)d\xi}\\
\\
&=&\ds{-u(t,x)+e^{x}\int^{\infty}_xe^{-\xi} m(t,\xi)d\xi}.
\ea
$$
Consequently, provided that $x\geq q(t,x_0)$ we have $u_x(t,x)\geq -u(t,x)$. Similarly,
$$
u_x(t,x)=u(t,x)-e^{-x}\int_{-\infty}^xe^{\xi} m(t,\xi)d\xi,
$$
which implies that $u_x(t,x)\geq u$, if $x\leq q(t,x_0)$.

Therefore, $u_x(t,x)\geq -\|u\|_{L^\infty(\R)}$. Since $\|u\|_{L^\infty(\R)}\leq \|u\|_{H^1(\R)}\leq \|u_0\|_{H^1(\R)}$, we arrive at $- \|u_0\|_{H^1(\R)}\leq u_x(t,x)$.
\hfill$\square$

\section{Criteria for occurrence of wave breaking}\label{Sec3}

In this section we will exhibit a criteria for the occurrence of wave breaking. In what follows, given an initial data $u_0\in H^3(\R)$, we will initially approach wave breaking by imposing some conditions to obtain an ordinary differential inequality for the $x$ derivative of the local solution $u\in C^1([0,T),H^2(\R))$. The bounds for this inequality will lead to a certain contradiction that will imply wave breaking.



We start with equation \eqref{gCH}. After applying the operator $\Lambda^{-2}$, it is possible to rewrite it in the evolutive form
\begin{align*}
    u_t + (u+\Gamma)u_{x}+u_x^2 = \Lambda^{-2}\partial_x h(u) - \Lambda^{-2}\partial_x\left(u^2+\frac{u_x^2}{2}\right),
\end{align*}
where $h$ is given by \eqref{h}.

After calculating the $x$ derivative and using the relation $\Lambda^{-2}\partial_x^2 = \Lambda^{-2}-1$, it is obtained
\begin{align}\label{utx}
    u_{tx}+(u+\Gamma)u_{x}+\frac{1}{2}u_x^2 = u^2 + \Lambda^{-2}h(u) - h(u) - \Lambda^{-2}\left(u^2+\frac{u_x^2}{2}\right).
\end{align}

Similarly to the global well-posedness proved in the previous section, we will also initially consider the criteria for wave breaking in $H^2(\R)$ and then extend the result to $s>3/2$ by the same density argument. We start with a lemma that will be of extreme importance.

\begin{lemma}\label{lem4.1}
Let $T>0$ and $v\in C^1([0,T),H^2(\R))$ be a given function. Then, for any $t\in[0,T)$, there exists at least one point $\xi(t)\in \R$ such that
\begin{align}\label{inf}
    y(t) :=\inf\limits_{x\in\R}v_x(t,x) = v_x(t,\xi(t))
\end{align}
and the function $y$ is almost everywhere differentiable in $(0,T)$, with $y'(t)=v_{tx}(t,\xi(t))$ almost everywhere in $(0,T)$.
\end{lemma}
\begin{proof}
See Theorem 2.1 in \cite{const1998-2} or Theorem 5 in \cite{Escher1}.
\end{proof}

Given $u_0\in H^3(\R)$, let $u\in C^1([0,T),H^2(\R))$ be the unique solution associated to the Cauchy problem of \eqref{cauchy}. Defining $$y(t) = \inf\limits_{x\in \R} u_x(t,x),$$ Lemma \ref{lem4.1} tells that, for any $t\in [0,T)$, there is a point $\xi(t)$ such that $u_x(t,\xi(t))$ reaches its minimum value. Evaluating \eqref{utx} at this point $(t,\xi(t))$ and observing that $u_{xx}(t,\xi(t))=0$, the following ordinary differential equation is obtained
\begin{align*}
    y'(t) + \frac{1}{2}y^2(t) = u^2(t,\xi(t)) - F(u(t,\xi(t))) + G(u),
\end{align*}
where $F(u) = \Lambda^{-2}\left(u^2(t,\xi(t))+\frac{1}{2}y^2(t)\right)$ and $G(u) = \Lambda^{-2} h(u(t,\xi(t))) - h(u(t,\xi(t)))$.

Since equation \eqref{cauchy} conserves the $H^1$-norm, a result from \cite{Escher1} (see page 106-107) tells that $$F(t)\geq \frac{1}{2}u^2(t,\xi(t)),\quad \forall t\in[0,T).$$ Thus, using again the results of \cite{Escher1} (see page 107), $u^2(t,\xi(t))\leq\Vert u_0\Vert^2_{H^1}/2$ and
$$u^2(t,\xi(t)) - F(u(t,\xi(t))) \leq \frac{1}{2} u^2(t,\xi(t)) \leq \frac{\Vert u_0\Vert^2_{H^1}}{4},$$
transforming the ODE into the differential inequality
\begin{align}\label{ODE}
    y'(t) + \frac{1}{2}y^2(t) \leq \frac{\Vert u_0\Vert^2_{H^1}}{4} + G(u).
\end{align}
We will now deal with the function $G(u)$.

\begin{proposition}\label{prop3.1}
    Let $u_0\in H^3(\R)$ and $u$ be the unique solution of \eqref{cauchy}. If $G(u) = \Lambda^{-2} h(u) - h(u)$, where $h$ is given by \eqref{h}, then there exists a positive integer $p$ such that 
    \begin{align*}
        \vert G(u)\vert \leq 9 K \Vert u_0\Vert^p_{H^1},
    \end{align*}
    where $K = 4\max\{\vert\alpha\vert,\vert\beta\vert/3,\vert\gamma\vert/4,\vert\Gamma\vert\}$.
\end{proposition}
\begin{proof}
For $h$ given as in \eqref{h}, we have
\begin{align*}
    \Vert h(u)\Vert_{H^1} &\leq (\vert\alpha\vert+\vert\Gamma\vert)\Vert u\Vert_{H^1} + \frac{\vert\beta\vert}{3}\Vert u\Vert_{H^1}^3 + \frac{\vert\gamma\vert}{4}\Vert u\Vert_{H^1}^4\\
    &= (\vert\alpha\vert+\vert\Gamma\vert)\Vert u_0\Vert_{H^1} + \frac{\vert\beta\vert}{3}\Vert u_0\Vert_{H^1}^3 + \frac{\vert\gamma\vert}{4}\Vert u_0\Vert_{H^1}^4\\
    &\leq K (\Vert u_0\Vert_{H^1}+\Vert u_0\Vert_{H^1}^3+\Vert u_0\Vert_{H^1}^4).
\end{align*}
Since $\Vert u\Vert_{L^\infty}\leq \Vert u\Vert_{H^1}$
\begin{align*}
    \vert h(u)\vert &\leq \Vert h(u)\Vert_{L^{\infty}} \leq \Vert h(u)\Vert_{H^1}\\
    &\leq K (\Vert u_0\Vert_{H^1}+\Vert u_0\Vert_{H^1}^3+\Vert u_0\Vert_{H^1}^4).
\end{align*}

On the other hand, we have
\begin{align*}
    \vert \Lambda^{-2}h(u)\vert &= \left\vert \int_{R} \frac{e^{-|x-y|}}{2}h(u(t,y))dy\right\vert\leq 2\Vert h(u)\Vert_{L^{\infty}} \\
    &\leq 2K (\Vert u_0\Vert_{H^1}+\Vert u_0\Vert_{H^1}^3+\Vert u_0\Vert_{H^1}^4).
\end{align*}
Thus
\begin{align*}
    \vert G(u)\vert &\leq \vert h(u)\vert + \vert \Lambda^{-2}h(u)\vert\\
    &\leq 3K (\Vert u_0\Vert_{H^1}+\Vert u_0\Vert_{H^1}^3+\Vert u_0\Vert_{H^1}^4).
\end{align*}
Let $p$ be such that $\Vert u_0\Vert_{H^1}^p = \max\{\Vert u_0\Vert_{H^1},\Vert u_0\Vert_{H^1}^3,\Vert u_0\Vert_{H^1}^4\}$. Then $ \vert G(u)\vert \leq 9K \Vert u_0\Vert_{H^1}^p$.
\end{proof}

Proposition \ref{prop3.1} is the last step towards our conditions for the existence of wave breaking. It is not sufficient, nevertheless, once the constants $C$ and $K$ play an important role on Proposition \ref{prop3.1}. We shall introduce a suitable parameter $\sigma$ to calibrate the presence of these constants.

\noindent\textbf{Proof of Theorem \ref{teo1.3}:} Let $\sigma\in \R$ be such that $$0<\sigma \leq \frac{1}{1+36K},$$ where $C$ and $K$ are the constants of Proposition \ref{prop3.1}. Moreover, assume that, given $u_0\in H^3(\R)$ there exists a point $x_0\in \R$ in a way that the condition
$$\sqrt{2\sigma} u_0'(x_0)< \min\{-\Vert u_0\Vert_{H^1}, - \Vert u_0\Vert_{H^1}^{p/2}\}$$
is satisfied, where $p$ is the power provided by Proposition \ref{prop3.1}. Choose $\epsilon \in (0,1)$ such that $$2\sigma (1-\epsilon)(u_0'(x_0))^2\geq (\min\{-\Vert u_0\Vert_{H^1}, - \Vert u_0\Vert_{H^1}^{p/2}\})^2.$$
On the one hand we have $y(0)\leq u_0'(x_0)<0$, where $y$ is given by \eqref{inf}, which means that $y^2(0)\geq (u_0'(x_0))^2$. This implies that 
\begin{align*}
    \Vert u_0\Vert_{H^1}^2 \leq 2\sigma (1-\epsilon)y^2(0),\quad 
    \Vert u_0\Vert_{H^1}^p \leq 2\sigma (1-\epsilon)y^2(0).
\end{align*}
Conversely, from \eqref{ODE} and Proposition \ref{prop3.1} we have
\begin{align}\label{ineq}
\begin{aligned}
    y'(t) + \frac{1}{2}y^2(t) &\leq \frac{1}{4}\Vert u_0\Vert_{H^1}^2 + 9K \Vert u_0\Vert_{H^1}^p \leq \frac{\sigma}{2}(1+36K)(1-\epsilon) y^2(0)\\ 
    &\leq \frac{1}{2}(1-\epsilon)y^2 (0).
\end{aligned}
\end{align}
The argument to finish the proof for this case is now reduced to the same one used for the proof of the Camassa-Holm equation (see \cite{Escher1} page 108). In summary, \eqref{ineq} implies that
$$y'(t)\leq - \frac{1}{4}\epsilon y^2(t),$$
which after using the Gronwall inequality, yields
$$0\geq \frac{1}{y(t)} \geq \frac{1}{y(0)} + \frac{1}{4}\epsilon t.$$
This means that the lifespan $T$ is finite and, according to Proposition \ref{prop2.3}, $u_x$ cannot be bounded from below. Therefore, there exists a $t_0$ such that $\lim\limits_{t\to t_0} y(t)=-\infty$ and the wave breaking occurs. This concludes the proof of Theorem \ref{teo1.3} for $s=2$, and, therefore, for $s>3/2$ by the same density argument used in the previous section.
\hfill$\square$

\section{Integrability results}\label{Sec4}

It is well known that integrable equations have a deep connection with geometry \cite{enriqueimrn,enriquejpa, reyes2002,artur,keti2015}. For this reason, in our previous work \cite{priigorjde} we investigated the relations of \eqref{cauchy} and pseudo-spherical surfaces \cite{reyes2002,keti2015}. Our results indicated that \eqref{cauchy} (with $\Gamma=0$) would only describe pseudo-spherical surfaces if $\beta=\gamma=0$, which is equivalent to say that the equations is essentially reduced to the CH equation. On the other hand, what we did in \cite{priigorjde} {\it suggest} that we could remove the condition $\Gamma=0$ so that the DGH would share the same property.

We observe that the investigation of the geometric integrability of \eqref{cauchy} was a first test for integrability properties of the equation. Moreover, integrability is a vast ocean of possibilities and there are several different definitions, such as  bi-Hamiltonian structure, infinitely many symmetries and/or conservation laws, or the existence of a Lax pair, see \cite{ablo1,ablo2,sok}. It is not our intention to make this work a treatment on integrability properties of equation \eqref{cauchy}, but it aims at investigating integrability of the equation from a different perspective.

An evolution equation $u_t=F$ is said to have a Hamiltonian structure if it can be written as 
$$
u_t=B\f{\de{\cal H}}{\de u},\quad {\cal H}=\int h(u)dx,\quad \f{\de{\cal H}}{\de u}:={\cal E}_u h,\quad {\cal E}_u:=\f{\p}{\p u}-\p_x\f{\p}{\p u_x}+\p_x^2\f{\p}{\p u_{xx}}+\cdots
$$
where ${\cal H}$ is a functional, $B$ is a Hamiltonian operator, and ${\cal E}_u$ is the Euler-Lagrange operator, see chapter 7 of Olver \cite{olverbook}. If the equation has two representations
\bb\label{4.0.1}
u_t=B_1\f{\de{\cal H}_2}{\de u}=B_2\f{\de{\cal H}_1}{\de u},
\ee
where ${\cal H}_1$ and ${\cal H}_2$ are two functionals, while $B_1$ and $B_2$ are Hamiltonian operators satisfying the Jacobi identity, which makes use of Poisson brackets, and such that its linear combination is still a Hamiltonian, then the equation is said to have bi-Hamiltonian structure, see chapter 7 of \cite{olverbook}.

We observe that equation \eqref{cauchy} can be written in two different representations (see also \cite{GLS} in which a similar representation\footnote{Actually, the representation \eqref{bi-hamil} is nothing but the one found in \cite{GLS} under suitable scalings and specific choices of the constants.} was firstly found to \eqref{chines-fluid}):
\bb\label{bi-hamil}
\ba{lcl}
m_t&=&\ds{-B_1\f{\de{\cal H}_2}{\de m}=-B_2\f{\de{\cal H}_1}{\de m}},\\
\\
    B_1 &=& \partial_x(1-\partial_x^2),\\
    \\
    B_2 &=& \ds{\partial_x\left(m\cdot\right) + m\partial_x -\al\p_x- \Gamma\partial_x^3 - \f{2}{3}\beta \partial_x\left(u\partial_x^{-1}(u\partial_x\cdot)\right)}\\
    \\
    &&\ds{- \f{5}{8}\gamma\partial_x \left(u^{3/2}\partial_x^{-1}(u^{3/2}\partial_x\cdot)\right)},\\
    \\
    {\cal H}_1&=&\ds{\f{1}{2}\int\left(u^2+u_x^2\right)dx},\\
    \\
    {\cal H}_2&=&\ds{\left(\f{3}{2}u^2-uu_{xx}-\f{1}{2}u_x^2-\al u-\f{\be}{3}u^3-\f{\gamma}{4}u^4-\Gamma u_{xx}\right)dx}.
    \ea
\ee

Above, $\p_x^{-1}f:=\int_{-\infty}^x f$, whereas $\p_x(m\cdot)f:=\p_x(mf)$ and so on.

Equation \eqref{bi-hamil} suggests that \eqref{cauchy} may have bi-Hamiltonian structure and, therefore, be integrable. However, the Jacobi identity, a necessary step for the proof of existence of a bi-Hamiltonian formulation, is not trivial at all to be checked \cite{DHH}. Moreover, we note that $B_2=B_{02}+B_{12}$, where $B_{12}$ is the part of the operator depending on $\be$ and $\gamma$ and can be interpreted as a deformation of $B_{02}$. We observe that $B_1$ and $B_{02}$ are Hamiltonian operators for the DGH equation \cite{DGH1,DGH2} and, according to \cite{DGH1}, no further deformations of the pair $(B_1,B_{02})$  (involving higher order derivatives) would be compatible with $B_{02}$ and, therefore, we do not believe that \eqref{bi-hamil} might be a bi-Hamiltonian structure to \eqref{cauchy}, unless $B_{12}\equiv0$, which is accomplished if and only if $\be=\gamma=0$, that is equivalent to returning to the DGH equation, and is well known to be integrable \cite{DGH1,DGH2}.

In view of the comments above and the previous results, we shall now prove that the DGH equation is also geometrically integrable, proving formally that the result suggested in \cite{priigorjde} can be extended to this equation. This is our first goal. The second one is to use the techniques introduced by Dubrovin in \cite{dubrovincmp}, see also \cite{dubrovincpam}, to investigate whether equation \eqref{cauchy} admits a perturbed bi-Hamiltonian structure, which would then imply on the existence of an infinite number of {\it approximate symmetries} \cite{dubrovincmp}.

\subsection{Geometric integrability of the Dullin-Gottwald-Holm equation}

We recall a very useful definition, see \cite{chern,keti2015}.

\begin{definition}\label{def4.1}
A differential equation for a real function $u=u(t,x)$ is said to describe pseudo-spherical surfaces if it is equivalent to the structure equations $d\omega_1=\omega_3\wedge\omega_2$, $d\omega_2=\omega_1\wedge\omega_3$ and $d\omega_3=\omega_1\wedge\omega_2$, of a 2-dimensional Riemmanian manifold whose Gaussian (or total) curvature is ${\cal K}=-1$ and its metric tensor is given by $g=\omega_1^2+\omega_2^2$.
\end{definition}

The description of pseudo-spherical surfaces is a peculiar property of an equation and, therefore, those describing one-parameter families of such surfaces warrant a definition (see Definition 2 in \cite{reyes2002}).
\begin{definition}\label{def4.2}
A differential equation is said to be geometrically integrable if it describes a non-trivial one-parameter family of pseudo-spherical surfaces.
\end{definition}

 In \cite{priigorjde} we replaced $u$ by $u-\Gamma$ in \eqref{cauchy}, which is equivalent to setting $\Gamma=0$ in \eqref{cauchy}. With this condition, we used the results established in \cite{keti2015} to prove that \eqref{cauchy}, with $\Gamma=0$, describes pseudo-spherical surfaces if and only if $\beta=\gamma=0$, that is, when the equation is reduced to the Camassa-Holm equation. More precisely, we proved that if
\bb\label{4.2.2}
\ba{l}
\ds{\omega_1=\left(m+b\right)dx-\left[um+(b+1)u+b\mp\eta u_x \right]}dt,\\
\\
\ds{\omega_2=\eta\,dx-[\eta(1+u)\mp u_x]dt},\\
\\
 \ds{\omega_3=\pm\left(m+b+1\right)dx+\left[\eta u_x\mp u(m+1) \mp(u+1)(b+1)\right]dt,}
\ea
\ee
with $b=-1+(\eta^2-\alpha)/2$ and $m=u-u_{xx}$, then equation \eqref{cauchy} arises as the compatibility condition of $d\omega_3=\omega_1\wedge\omega_2$. For further details, see \cite{priigorjde,reyes2002,keti2015}.


Our results proved in \cite{priigorjde} {\it suggest} that \eqref{cauchy} would also describe pseudo-spherical surfaces for any value of $\Gamma$. One can argue that we could obtain this result simply replacing $u$ by $u+\Gamma$ in \eqref{4.2.2}, but such a change, although geometrically viable, is not suitable if we want to preserve the solutions of the equation. We note that the solutions we are dealing with in the present work vanish at infinity, but the same does not occur to $u+\Gamma$ unless one assumes $\Gamma=0$.

Our strategy here is to make a change of coordinates preserving the initial data and, from the $1-$forms \eqref{4.2.2}, prove that \eqref{cauchy} describes pseudo-spherical surfaces to any values of $\alpha$ and $\Gamma$ when $\beta=\gamma=0$. We note that this is equivalent to prove that the Dullin-Gotwald-Holm equation describes pseudo-spherical surfaces. Since the Camassa-Holm equation \eqref{CH} is known to be geometrically integrable, we will complete the scheme showing that the Dullin-Gottwald-Holm equation, that is, $\beta=\gamma=0$ and $\Gamma\neq 0$ in \eqref{cauchy}, is integrable in this sense of Definition \ref{def4.2}.

Let us consider the function $F:\R^2\rightarrow\R^2$, defined by 
\bb\label{4.2.3}
F(t,x)=(t,x-\Gamma t).
\ee
It is a global diffeomorphism and, therefore, we can interpret it as a change of coordinates. Let $(\tau,\chi):=F(t,x)$. Note that $F(0,x)=(0,x)$, that is, any point of $\{0\}\times\R$ is a fixed point of the diffeomorphism $F$.

Let $u$ be a solution of the equation
\bb\label{4.2.4}
m_t + um_x + 2u_xm = \alpha u_x 
\ee
and define $u=v\circ F$. Substituting this into \eqref{4.2.4}, we obtain
\bb\label{4.2.5}
n_\tau + vn_\chi + 2v_\chi n = \tilde{\alpha} v_\chi+\Gamma v_{\chi\chi\chi},
\ee
where $n:=v-v_{\chi\chi}$ and $\tilde{\al}=\al+\Gamma$. Moreover, we observe that if $v$ is a solution of \eqref{4.2.5}, then $u=v\circ F$ is a solution of \eqref{4.2.4}. Conversely, if $u$ is a solution of \eqref{4.2.4}, then $v=u\circ F^{-1}$, where $F^{-1}$ denotes the inverse of the function \eqref{4.2.3}, is a solution of \eqref{4.2.5}.

We now observe that the pullback (see Lemma 14.16 on page 361 of \cite{lee}) of the basis of $1-$forms $\{dt,dx\}$ is $F^\ast(dt)=d\tau$ and $F^\ast(dx)=d\chi-\Gamma d\tau $. Let $\theta_i=F^\ast\omega_i$, $i=1,2,3$. Then the triplet \eqref{4.2.2} changes to
\bb\label{4.2.6}
\ba{l}
\ds{\theta_1=\left(n+b\right)d\chi-\left[(v+\Gamma)n+(b+1)v+b(\Gamma+1)\mp\eta v_\chi \right]}d\tau,\\
\\
\ds{\theta_2=\eta\,d\chi-[\eta(1+v+\Gamma)\mp v_\chi]d\tau},\\
\\
 \ds{\theta_3=\pm\left(n+b+1\right)d\chi+\left[\eta v_\chi\mp (v+\Gamma)(n+1) \mp(v+1)(b+1)\mp \Gamma b\right]d\tau,}
\ea
\ee
where $\eta^2=2+2b+\al+\Gamma$.

\noindent\textbf{Proof of Theorem \ref{teo1.4}:} Let us first prove that if $u\in C^{k+2}(\R)$ is a solution of \eqref{4.2.5}, then \eqref{4.2.6} corresponds to the desired triplet. We shall prove to the case
$$
\ba{l}
\ds{\theta_1=\left(n+b\right)d\chi-\left[(v+\Gamma)n+(b+1)v+b(\Gamma+1)-\eta v_\chi \right]}d\tau,\\
\\
\ds{\theta_2=\eta\,d\chi-[\eta(1+v+\Gamma)- v_\chi]d\tau},\\
\\
 \ds{\theta_3=\pm\left(n+b+1\right)d\chi+\left[\eta v_\chi- (v+\Gamma)(n+1)-(v+1)(b+1)- \Gamma b\right]d\tau,}
 \ea
$$
The other choice of sign is similar and can be checked by the same way. On the one hand, we have
\bb\label{4.2.7}
d\theta_3=-\left(n_\tau+v_{\chi}n+vn_\chi-\Gamma v_{\chi\chi\chi}-\eta v_{\chi\chi}+(b+1-\Gamma)v_\chi\right)d\chi\wedge d\tau.
\ee
On the other hand,
\bb\label{4.2.8}
\theta_1\wedge\theta_2=\left(v_\chi n+v_{\chi\chi}\eta +(b-\eta^2)v_\chi\right)d\chi\wedge d\tau.
\ee
Equating \eqref{4.2.7} to \eqref{4.2.8} and taking the relation $\eta^2=2+2b+\al+\Gamma$ into account we obtain $n_\tau + vn_\chi + 2v_\chi n = \alpha v_\chi+\Gamma v_{\chi\chi\chi}$.
\hfill$\square$




\subsection{Bi-Hamiltonian deformations}\label{bi-ha}

Here we apply another test of integrability to \eqref{cauchy} following the ideas presented in \cite{dubrovincmp}. Our main ingredient is the hyperbolic equation
\bb\label{4.1.1}
v_t+a(v)v_x=0,
\ee
which is well-known to have infinitely many symmetries and conservation laws ({\it e.g}, see \cite{igorjnmp,igorcoam} and references therein) and has the representation
\bb\label{4.1.2}
v_t+\{v(x),H_0\}=0,\quad H_0=\int f(v)dx,\quad f''(v)=a(v).
\ee
In \eqref{4.1.2} ${\cal H}_0$ is the unperturbed Hamiltonian, while $\{\cdot,\cdot\}$ denotes a Poisson bracket.

The approach proposed by Dubrovin \cite{dubrovincpam,dubrovincmp} assures the existence and uniqueness of bi-Hamiltonian representations of perturbations of \eqref{4.1.1}. Our use of this approach relies on some observations:
\begin{itemize}
    \item Well-known integrable equations satisfy the conditions in \cite{dubrovincmp};
    \item The Camassa-Holm equation is among the equations investigated by Dubrovin. Then, since \eqref{cauchy} is a generalisation of the mentioned equation, it is natural to check what would be the result of Dubrovin's approach applied to \eqref{cauchy};
    \item We have the uniqueness of a deformed bi-Hamiltonian structure, if it exists.
\end{itemize}

We begin our analysis recalling some results proved in \cite{dubrovincmp}. Below, $\{\cdot,\cdot\}_1$ and $\{\cdot,\cdot\}_2$ denote two different Poisson brackets.

\begin{lemma}\label{lema4.1}
For any $f=f(v)$ the Hamiltonian flow
\bb\label{4.1.3}
\ba{lcl}
&&\ds{v_t+\p_x\f{\delta {\cal H}_f}{\delta v}}=0,\\
\\
\ds{{\cal H}_f}&=&\ds{\int h_f(v)\,dx},\\
\\
\ds{h_f}&=&\ds{f-\epsilon^2\f{c}{24}f'''v_x^2+\epsilon^4\left[\left(pf'''+\f{c}{480}f^{(4)}\right)v_{xx}^2\right.}\\
\\
&&\ds{\left.-\left(\f{cc''}{1152}f^{(4)}+\f{cc'}{1152}f^{(5)}+\f{c^2}{3456}f^{(6)}\right)+\f{p'}{6}f^{(4)}+\f{p}{6}f^{(5)}\right]},
\ea
\ee
where $c=c(v),\,\,p=p(v)$ is a symmetry, modulo ${\cal O}(\epsilon^6)$ of \eqref{4.1.2}. Moreover, the Hamiltonian $H_f$ commute pairwise, in the sense that $\{H_f,H_g\}={\cal O}(\epsilon^6)$ for arbitrary functions $f$ and $g$.
\end{lemma}
\begin{proof}
See Lemma 2.3 in \cite{dubrovincmp}.
\end{proof}
Not all perturbations are bi-Hamiltonian \cite{dubrovincmp}, but we have a very nice result.
\begin{lemma}\label{lema4.2}
If $c(u)\neq0$, the corresponding Hamiltonians \eqref{4.1.3} admit a unique bi-Hamiltonian structure obtained by the deformation of
$$
\{v(x),v(y)\}_1=\delta'(x-y),\quad \{v(x),v(y)\}_2=q(v(x))\delta'(x-y)+\f{1}{2}q'(v)v_x\delta(x-y)
$$
with the functions $p$, $q$ and $c$ satisfy the relation
\bb\label{4.1.4}
p(v)=\f{c(v)^2}{960}\left(5\f{c'(v)}{c(v)}-\f{q''(v)}{q'(v)}\right)
\ee
and $s(v)=0$.
\end{lemma}
\begin{proof}
See Theorem 4.1 in \cite{dubrovincmp}.
\end{proof}

The uniqueness assured by Lemma \ref{lema4.2} motivated us to propose Definition \ref{new}. The key to investigate whether \eqref{cauchy} satisfies Definition \ref{new} is to consider its higher order terms in the derivatives of $u$ as perturbations of a hyperbolic equation of the type \eqref{4.1.1}. Therefore, let us make the change $(t,x,u)\mapsto(\epsilon t,\epsilon x,u/\sqrt{2})$, $0\neq\epsilon\ll 1$ into the variables in \eqref{cauchy}, which then yields (after renaming constants) equation \eqref{4.1.5}, namely

\begin{align*}
    u_t-\epsilon^2u_{txx}-\left(\f{3}{2}uu_x+\al u_x+\be u^2u_x+\gamma u^3u_x\right)+\epsilon^2\left(u_xu_{xx}+\f{1}{2}uu_{xxx}-\Gamma u_{xxx}\right)=0.
\end{align*}
Observe that we re-scaled equation \eqref{cauchy} in the form \eqref{4.1.5} in order to be able to compare our results with those obtained by Dubrovin regarding the CH equation.

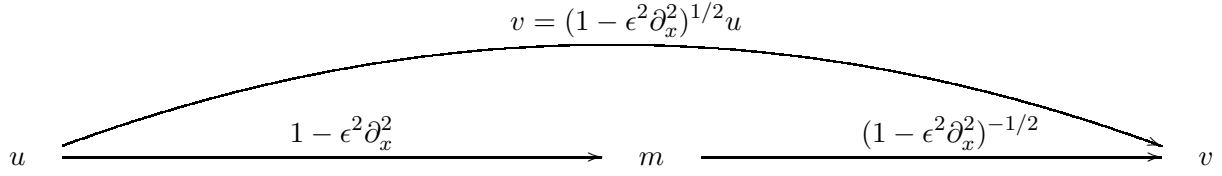
\begin{figure}[ht!]
\centering
\xymatrix{ \ds{u} \quad\ar@/^1.5cm/[rrrrrrrrrrrrr]^{\quad\ds{v=(1-\epsilon^2\p_x^2)^{1/2}}u}\ar[rrrrrrr]^{\ds{1-\epsilon^2\p_x^2}}& & & & & & &\quad \ds{m}\quad \ar[rrrrrr]^{\quad\quad \ds{(1-\epsilon^2\p_x^2)^{-1/2}}}& & & & & &\quad \ds{v}}
\caption{The diagram shows the sequence of transformations needed to transform \eqref{4.1.5} into a perturbed evolution equation.}\label{fig1}
\end{figure}

We observe that lemmas \ref{lema4.1} and \ref{lema4.2} are only applicable to evolution equations and \eqref{4.1.5} it is not initially presented as one. However, we note that $u-\epsilon^2\p_x^2u=(1-\epsilon^2\p_x^2)u$ and, since it is a perturbation of the identity operator, it is invertible and, for $s\in\R$,
\bb\label{4.1.6}
(1-\epsilon^2\p_x^2)^{s/2}=1+s\epsilon^2\p_x^2+s(s+1)\epsilon^4\p_x^4+{\cal O}(\epsilon^6)\approx 1+s\epsilon^2\p_x^2+s(s+1)\epsilon^4\p_x^4.
\ee

It is worth mentioning that this operator commutes with $\p_t$. We now transform the non-evolution equation \eqref{4.1.5} into an evolution one through the following steps:
\begin{itemize}
    \item We firstly identify the term $u_t-\epsilon^2u_{txx}$ as $m$;
    \item In the remaining terms we substitute $u$ by $u=(1-\epsilon^2\p_x^2)v$, $u_x=\p_x\left((1-\epsilon^2\p_x^2)v\right)$ and $u_{xx}=\p_{x}^2\left((1-\epsilon^2\p_x^2)v\right)$;
    \item We use the expansions 
    $$
    \ba{lcl}
    u&=&\ds{v+\f{\epsilon^2}{2}v_{xx}+\f{3}{8}\epsilon^4v_{xxx}}+\cdots,\quad
    u_x=\ds{v_x+\f{\epsilon^2}{2}v_{xxx}+\f{3}{8}\epsilon^4v_{xxxx}}+\cdots,\\
    \\
    u_{xx}&=&\ds{v_{xx}+\f{\epsilon^2}{2}v_{xxxx}+\f{3}{8}\epsilon^4v_{xxxxx}}+\cdots,\quad
    u_{xxx}=\ds{v_{xxx}+\f{\epsilon^2}{2}v_{xxxxx}+\f{3}{8}\epsilon^4v_{xxxxxx}}+\cdots
    \ea
    $$
    \item We then apply \eqref{4.1.6} with $s=-1/2$ to the resulting equation (see Figure \ref{fig1}).
\end{itemize}
The final result of the scheme above is the equation
\bb\label{4.1.7}
\ba{lcl}
v_t&=&-F_{00}-H_0-\epsilon^2\left[F_{01}-F_{10}+H_1+\f{1}{2}\p_x^2\left(F_{00}+H_0\right)\right]\\
\\
&&-\epsilon^4\left[F_{02}-F_{11}+H_2+\f{1}{2}\p_x^2\left(F_{01}-F_{10}+H_1\right)+\f{3}{8}\p_{x}^4\left(F_{00}+H_0\right)\right]+{\cal O}(\epsilon^6),
\ea
\ee
whose terms are
\bb\label{4.1.8}
\ba{lcl}
F_{00}&=&\ds{\f{3}{2}vv_x},\quad \ds{F_{01}=\f{3}{4}(vv_{xxx}+v_xv_{xx})},\quad F_{10}=\ds{\f{1}{2}\left(vv_{xxx}+2v_xv_{xx}\right)},\\
\\
F_{11}&=&\ds{\f{1}{4}\left(vv_{xxxxx}+2v_xv_{xxxx }+3v_{xx}v_{xxx}\right)},\\
\\
F_{02}&=&\ds{\f{3}{16}\left(3vv_{xxxxx}+3v_xv_{xxxx}+2v_{xx}v_{xxx}\right)},\\
\\
H_0&=&\ds{\al v_x+\beta v^2 v_x+\gamma v^3 v_x},\\
\\
H_1&=&\ds{\Gamma v_{xxx}+\f{\al}{2}v_{xxx}+\f{\beta}{2}\left(v^2v_{xxx}+2vv_xv_{xx}\right)+\f{\gamma}{2}\left(3v^2v_xv_{xx}+v^3v_{xxx}\right)}\\
\\
H_2&=&\ds{\left(\f{\Gamma}{2}+\f{3}{8}\alpha\right)v_{xxxxx}+\f{\beta}{8}\left(3v^3v_{xxxxx}+4vv_{xx}v_{xxx}+6vv_xv_{xxxx}+2v_xv_{xx}^2\right)}\\
\\
&&\ds{+\f{\gamma}{8}\left(3v^3v_{xxxxx}+9v^2v_xv_{xxxx}+3vv_xv_{xx}^2\right)}.
\ea
\ee
From now on, we neglect terms ${\cal O}(\epsilon^6)$. 
The Hamiltonian in \eqref{4.1.3} is of the form
\bb\label{4.1.9}
{\cal H}_f={\cal H}_0+\epsilon^2{\cal H}_1+\epsilon^4{\cal H}_2,
\ee
where
\bb\label{4.1.10}
\ba{lcl}
{\cal H}_0&=&\ds{\int f(v)dx},\quad
{\cal H}_1=\ds{-\int \f{c}{24}f'''v_x^2dx},\\
\\
{\cal H}_2&=&\ds{\int\left[\left(pf'''+\f{c}{480}f^{(4)}\right)v_{xx}^2-\left(\f{cc''}{1152}f^{(4)}+\f{cc'}{1152}f^{(5)}+\f{c^2}{3456}f^{(6)}\right)+\f{p'}{6}f^{(4)}+\f{p}{6}f^{(5)}\right]dx},
\ea
\ee

Our aim here is to investigate if equation \eqref{4.1.8} can be written in the form \eqref{4.1.3}. Then, assuming this hypothesis, from \eqref{4.1.8}, \eqref{4.1.10} and the first equation in \eqref{4.1.3} we have the following equations, from the coefficients of $1$, ${\cal O}(\epsilon^2)$ and ${\cal O}(\epsilon^4)$, respectively,
\bb\label{4.1.11}
F_{00}+H_0=\p_x\f{\de{\cal H}_0}{\de v},
\ee
\bb\label{4.1.12}
F_{01}-F_{10}+H_1+\f{1}{2}\p_x^2\left(F_{00}+H_0\right)=\p_x\f{\de{\cal H}_1}{\de v},
\ee
\bb\label{4.1.13}
F_{02}-F_{11}+H_2+\f{1}{2}\p_x^2\left(F_{01}-F_{10}+H_1\right)+\f{3}{8}\p_x^4\left(F_{00}+H_0\right)=\p_x\f{\p{\cal H}_2}{\de v}.
\ee

We now use the equation \eqref{4.1.11} to determine whether, and under which conditions, if any, \eqref{cauchy} satisfies Lemma \ref{lema4.1}. From now on we take all constants of integration as $0$.

Using the first equation in \eqref{4.1.10}, $F_{00}$ and $H_0$ from \eqref{4.1.8} and \eqref{4.1.11} we conclude that 
\bb\label{4.1.14}
f(v)=\f{v^3}{4}+\f{\al}{2}v^2+\f{\be}{12}v^4+\f{\gamma}{20}v^5. 
\ee

Substituting \eqref{4.1.13} into ${\cal H}_1$ in \eqref{4.1.10} and using $F_{00}$, $F_{01}$, $F_{10}$, $H_0$ and $H_1$ from \eqref{4.1.8}, equation \eqref{4.1.12} reads
$$
\ba{lcl}
\ds{\left(\f{\gamma}{4}cv^2+\f{\be}{6}c v+\f{c}{8}\right)v_{xx}}&+&\ds{\left[\f{c'}{16}+\f{\be}{12}c+\left(\f{\gamma}{4}c+\f{\be}{12}c'\right)v+\f{\gamma}{8}c'\right]v_x^2}\\
\\&=&\ds{\left(\gamma v^3+\be v^2+v+\Gamma+\al\right)v_{xx}+\left(\f{1}{2}+2\be v+3\gamma v^2\right)v_{x}^2}.
\ea
$$
The coefficients of $v_{xx}$ and $v_x^2$ yield the following system, respectively,
\bb\label{4.1.15}
\left(\f{\gamma}{4}v^2+\f{\be}{6} v+\f{1}{8}\right)c=\gamma v^3+\be v^2+v+\Gamma+\al
\ee
and
\bb\label{4.1.16}
\left(\f{1}{16}+\f{\be}{12}v+\f{\gamma}{4}v^2\right)c'+\left(\f{\gamma}{4}v+\f{\be}{12}\right)c'=\f{1}{2 }+2\be v+3\gamma v^2.
\ee
Deriving \eqref{4.1.15} with respect to $v$, multiplying the result by $-1/2$ and summing with \eqref{4.1.16}, we conclude that
$$
\be=\gamma=0.
$$
Substituting these values into \eqref{4.1.15} we conclude that $c(v)=8(v+\al+\Gamma)$.

Substituting $\be=\gamma=0$ into \eqref{4.1.14}, using the function $c$, we obtain the following functional from \eqref{4.1.10}
$$
{\cal H}_2=\f{3}{2}\int\left(p(v)v_{xx}^2-s(v)v_x^4\right)dx.
$$
Following the same steps, from the coefficients of $v_{xxxx}$,  $v_x^2v_{xx}$ in result of \eqref{4.1.13} we obtain
$$
3p(v)=v+\al+\Gamma\quad\text{and}\quad 18s(v)+3p''(v)=0,
$$
which implies that $p=(v+\al+\Gamma)/3$ and $s(v)=0$. Note that we can determine $q$ so that Lemma \ref{lema4.2} is satisfied. It is enough to take $q(v)=av+b$, where $a$ and $b$ are two arbitrary constants, with $a\neq0$. This proves Theorem \ref{teo1.5}.

\section{Discussion and Conclusion}\label{Sec5}
In \cite{priigorjde} we proved the local well-posedness of equation \eqref{cauchy}, but it was not possible to infer what sort of behaviour the solution $u$ would have as time approaches the maximal time of existence. More explicitly, we were not able to guarantee existence of global solutions or possibility of blow-up.

Here we established conditions to the existence of global solutions provided that $u_x$ is bounded from below, which is a result consistent with the literature of the CH equation \cite{CE,Escher1}, see also \cite{wujde2009}. We also proved the existence of wave breaking, which again is similar to the CH equation \cite{CE,const1998-2, Escher1}. The main difference between our case and the other is the presence of cubic and quartic nonlinearities, which bring significantly challenges to our proofs, see also \cite{igor}. We observe that conditions to ensure that $\sign{(m)}=\sign{(m_0)}$ have not been found for general forms of the function $h(u)$ in \eqref{h} and it is worth of further investigation.

We observe that of fundamental importance to these demonstrations is the fact that \eqref{cauchy} conserves the $H^1(\R)$-norm, see \cite{priigorjde}, which is a property shared not only with the CH equation, but also with the DGH, Novikov and other similar equations, see \cite{anco,boz,CH,DGH1,DGH2,pri-book,aims, cnmac, priigorjde,HW}.

We note that in \cite{GLS} local existence and wave breaking analysis was considered for the rotation-Camassa-Holm equation \eqref{chines-fluid} in a different way and their conditions are different of those we requested in Theorem \ref{teo1.3}. Compare, for example, Theorem 4.2 of \cite{GLS} with Theorem \ref{teo1.3} of the present work. 

Following our investigation about integrable properties of equation \eqref{cauchy}, we extended the geometric integrability of the CH equation \cite{reyes2002,enriquejpa,enriqueimrn,priigorjde} to the DGH equation. 

Equation \eqref{cauchy}, as mentioned in the Introduction, includes the rotation-Camassa-Holm equation \eqref{chines-fluid}, which has recently proposed \cite{chines-adv,gui-jnl,GLS,chines-jde} as a model for shallow water regime under the Coriolis effect. In \cite{GLS} the authors found two representations (which, up to scaling, is the one given in \eqref{4.1.10}), which suggests the possibility of a bi-Hamiltonian structure to the model.

According to the literature and our own results, we do not believe that equation \eqref{cauchy} has a bi-Hamiltonian structure unless $\be=\gamma=0$, which also implies that \eqref{chines-fluid} would have a bi-Hamiltonian structure if and only if its nonlinearities are of order $2$, that is the same to say that it becomes the DGH equation. Evidence supporting our point is:
\begin{enumerate}
    \item In \cite{priigorjde} we investigated the geometric integrability of \eqref{cauchy} without $\Gamma$. As a consequence, we conclude that such equation would be geometrically integrable only when it is reduced to the CH equation.
    \item In the present work, more specifically in Theorem \ref{teo1.4}, we showed that the DGH equation is also geometrically integrable. These two results suggest that \eqref{cauchy} is only integrable when $\be=\gamma=0$.
    \item Based on the results in \cite{dubrovincmp}, we proposed the definition of quasi-integrability (see Definition \ref{new}) and we showed that \eqref{cauchy} is only quasi-integrable when it is reduced to the DGH equation, which is a well-known integrable equation \cite{DGH1}.
\end{enumerate}
We note that the two integrability tests we applied to \eqref{cauchy} only indicate integrability when the equation becomes the DGH equation. 

In \cite{DGH1} the authors mention that no (higher order) perturbations of the operator $B_2$ with $\be=\gamma=0$ in \eqref{bi-hamil} would make the resulting operator, jointly with $B_1$, a Hamiltonian pair. Therefore, we believe that the representation in \eqref{bi-hamil} are two Hamiltonian representations, but does not bring a bi-Hamiltonian structure to \eqref{cauchy} if either $\be=0$ or $\gamma=0$. We observe, however, that to give a definitive answer whether \eqref{bi-hamil} is a genuine bi-Hamiltonian structure, we should prove that $B_2$ satisfy the Jacobi identity\footnote{Operator $B_1$ in \eqref{bi-hamil}, being constant, satisfies all conditions required to be such an operator.} and then both $B_1$ and $B_2$ would be compatible \cite{olverbook}. Unfortunately we have not succeeded in these points. In \cite{GLS} the authors exhibited a pair of operators similar to the ones in \eqref{bi-hamil} to the model \eqref{chines-fluid}. They called the pair {\it formal bi-Hamiltonian structure}, although it was not claimed that it is really a bi-Hamiltonian structure.

Last but not least, we would like to observe that in \cite{hay} another definition of integrability was proposed motivated by the works of Dubrovin and co-authors \cite{dubrovincpam,dubrovincmp}, and it is based on the existence of commuting formal series (see \cite{hay}, Section 3). Such definition is more focused on the symmetries, while in our case the existence of symmetries will be a consequence of the existence of a bi-Hamiltonian structure (see Lemma \ref{lema4.2}). We strongly believe that our equation would be integrable in the sense proposed by Hay {\it et. al.} under the same restrictions we found.

\section*{Acknowledgements}

The second author would like to express his gratitude to the Mathematical Institute of the Silesian University in Opava for the warm hospitality and work atmosphere found there during his visit. Particular thanks are given to Professor A. Sergyeyev for his kindness and discussions.

The work of I. L. Freire is supported by CNPq (grants 308516/2016-8 and 404912/2016-8).

\end{document}